\pdfoutput=1
\documentclass[twocolumn, 10pt]{article}
\usepackage[utf8]{inputenc}
\usepackage[T1]{fontenc}
\usepackage{amsmath,amssymb,amsthm,mathtools}
\usepackage{bm}
\usepackage{authblk}
\usepackage{hyperref}
\usepackage{geometry}
\usepackage{graphicx}
\usepackage{enumitem}
\usepackage{float}
\DeclarePairedDelimiter\ip{\langle}{\rangle}
\numberwithin{equation}{section}
\theoremstyle{plain}
\newtheorem{theorem}{Theorem}[section]
\newtheorem{lemma}[theorem]{Lemma}
\newtheorem{proposition}[theorem]{Proposition}
\newtheorem{corollary}[theorem]{Corollary}
\theoremstyle{definition}

\hypersetup{
    hidelinks
    }

\title{From Kinetic Nonequilibrium to Deviatoric Stress: An Operator Formulation of the First-Order Chapman--Enskog Transport Map}

\author{Tristan Barkman}

\date{\vspace{-5ex}}

\begin{document}
\twocolumn[
\maketitle

\begin{abstract}
The Chapman--Enskog procedure provides the classical asymptotic connection between kinetic nonequilibrium and macroscopic viscous stress. The present work isolates its first-order kinetic-to-stress step as an explicit operator composition. The construction separates the leading-order streaming source, the collisional kinetic response, and the deviatoric second-moment projection into a single transport map. Under explicit solvability, coercivity and bounded-inverse, regularity, and controlled-remainder assumptions, this map is well defined and bounded; hypocoercive estimates may additionally be used in closing the conditional remainder bound. In particular, a nonzero first-order deviatoric stress requires a nonzero first-order kinetic correction, although the converse is not asserted. The result is not a new derivation of the Navier--Stokes constitutive law or viscosity coefficient; its contribution is to make the kinetic-to-macroscopic transport map and its analytic conditions explicit. A BGK example recovers the classical Newtonian constitutive relation.
\end{abstract}

\vspace{2em}
]

\section{Introduction}

A central problem in statistical mechanics and fluid dynamics is the passage from microscopic particle dynamics to macroscopic transport laws. At the microscopic level, a fluid is described by interacting particles; at the macroscopic level, it is described by fields such as density, velocity, temperature, and stress. Between these descriptions lies kinetic theory, where the single-particle distribution function retains velocity-space information while its low-order moments define the hydrodynamic variables. Mathematical treatments of this microscopic-to-macroscopic passage include direct particle-to-hydrodynamic limits as well as the kinetic route through the Boltzmann equation and its hydrodynamic limits \cite{MicroscopicHydro1967,Gallagher2019,BardosGolseLevermoreI,BardosGolseLevermoreII,SaintRaymond}.

The Chapman--Enskog construction is one of the principal mechanisms by which the kinetic level is connected to continuum constitutive transport. Starting from the Boltzmann equation, one expands about a local Maxwellian in a small Knudsen number and obtains, at successive orders, corrections to the equilibrium distribution. The first correction produces the viscous stress and heat flux, leading to the Navier--Stokes--Fourier constitutive laws (chapter 7 of \cite{ChapmanCowling}, sections 2.8, 5.3, 5.6 of \cite{Cercignani}). {The first-order Chapman--Enskog calculation converts kinetic nonequilibrium information into macroscopic transport coefficients and fluxes.}

The purpose of the present paper is to isolate and formalize this kinetic-to-macroscopic interface. The main object is not a new viscosity coefficient, nor a replacement for Chapman--Enskog theory. In the nondimensional variables introduced in Section~\ref{sec:model}, I define the composite operator
\begin{equation}
\label{eq:kinetic-stress-map-intro}
{\mathcal{K}:=\mathcal{M}_2^{\mathrm{dev}}L^{-1}(I-\Pi),}
\end{equation}
where $\mathcal{S}=\partial_t^{(0)}+v\cdot\nabla_x$ is the leading-order streaming operator acting on the local Maxwellian, $L^{-1}$ is the pseudoinverse of the linearized collision operator on the orthogonal complement of its collision-invariant nullspace, and $\mathcal{M}_2^{\mathrm{dev}}$ extracts the deviatoric second velocity moment. The first-order stress is then represented as
\begin{equation}
\label{eq:kinetic-stress-map}
\tau^{(1)}=\mathcal{K}[\mathcal{S}f^{(0)}].
\end{equation}
This factorization makes the kinetic-to-macroscopic passage explicit: macroscopic gradients generate a kinetic source, collisions determine the corresponding kinetic response, and the second velocity moment converts that response into deviatoric momentum transport.

{This perspective clarifies the necessity statement:} if the first Chapman--Enskog correction vanishes, then its stress moment vanishes,
\begin{equation}
\label{eq:necessity-intro}
f^{(1)}\equiv0\quad\Longrightarrow\quad\tau^{(1)}\equiv0.
\end{equation}
The converse is not generally true, because a nonzero kinetic correction may have zero projection onto the deviatoric second-moment subspace. The substantive object is therefore the complete operator chain $\mathcal{K}\mathcal{S}$, together with the conditions that make each step well posed and the remainder estimate that prevents higher-order terms from contaminating the $O(\varepsilon)$ constitutive level.

{The present work addresses a specific part of the broader microscopic-to-macroscopic problem: it begins after the kinetic description has been established.} It does not derive the Boltzmann equation from Newtonian particle dynamics or identify a previously measured molecular velocity fluctuation with $f^{(1)}$. Instead, it addresses how a kinetic departure from local equilibrium is mapped into a continuum stress. This distinction matters because the microscopic-to-kinetic and kinetic-to-macroscopic passages involve different mathematical operations and sources of approximation \cite{Gallagher2019,BardosGolseLevermoreI,BardosGolseLevermoreII}.

{This work is motivated in part by the author's preceding studies of molecular collision statistics.} In \cite{Barkman}, molecular collisions were used to construct a discrete stochastic model for finite-scale velocity fluctuations, with an explicit dependence of fluctuation amplitude on the coarse-graining scale. A subsequent study {represented collisions by a compound Poisson process}, established linear variance growth and a Gaussian diffusion limit, and examined the response under a linear hydrodynamic relaxation model \cite{Barkman2026b}. Those studies address the statistical generation and propagation of molecular-scale fluctuations. The present paper treats the distinct kinetic step between this microscopic statistical structure and continuum constitutive stress.

{The research sequence across these studies is}
\begin{equation}
\boxed{
\begin{array}{c}
\text{molecular collisions}
\\[2pt]
\downarrow
\\[2pt]
\text{finite-scale fluctuations}
\\[2pt]
\downarrow
\\[2pt]
\text{kinetic departure from equilibrium}
\\[2pt]
\downarrow
\\[2pt]
\text{deviatoric stress}
\\[2pt]
\downarrow
\\[2pt]
\text{macroscopic transport}
\end{array}}
\end{equation}
The first two arrows are addressed statistically in the preceding work; the present paper develops the third and fourth arrows at the kinetic-to-macroscopic level. A direct identification of a particular molecular fluctuation with a Chapman--Enskog coefficient, or a complete derivation of nonlinear transition and turbulence from molecular events, remains outside the scope of the present analysis.

{Under explicit functional-analytic hypotheses, the first Chapman--Enskog constitutive correction is a bounded composition of (i) macroscopic-to-kinetic streaming, (ii) inversion of the dissipative collision operator, and (iii) kinetic-to-macroscopic moment projection.} This gives a mathematically explicit interface within the larger microscopic-to-macroscopic transport problem while retaining the classical Chapman--Enskog constitutive content.

\section{Previous work}

The passage from microscopic particle dynamics to hydrodynamic equations has a long history and remains an active mathematical problem. Direct microscopic derivations and scaling limits connect interacting particle systems to continuum equations, while the Boltzmann equation provides an intermediate kinetic description in which hydrodynamic variables arise as velocity moments \cite{MicroscopicHydro1967,Gallagher2019}. Within kinetic theory, the Chapman--Enskog and Hilbert constructions provide systematic low-Knudsen expansions leading to dissipative hydrodynamics \cite{ChapmanCowling,Cercignani}. Rigorous fluid-dynamic limits of the Boltzmann equation were developed by Bardos, Golse, and Levermore and subsequently extended in a broad mathematical literature \cite{BardosGolseLevermoreI,BardosGolseLevermoreII,SaintRaymond}.

{
The distinction from these hydrodynamic-limit works is one of objective and scale of analysis. Bardos--Golse--Levermore and Saint-Raymond study asymptotic regimes in which kinetic solutions converge to macroscopic fluid equations, addressing the derivation and validity of those limiting descriptions. The present paper does not establish such a limit or propose a competing derivation. Instead, taking the first-order Chapman--Enskog setting as its starting point, it isolates the stress-producing correction as the composition {$\mathcal{K}\mathcal{S}=\mathcal{M}_2^{\mathrm{dev}}L^{-1}(I-\Pi)\mathcal{S}$} and states conditions for this map to be well posed, bounded, and separated from the higher-order remainder. Thus the contribution is an explicit operator-level formulation of a familiar constitutive step, not a new hydrodynamic-limit result or transport law.
}

The kinetic-to-macroscopic step is therefore well established in its classical form. In particular, the first Chapman--Enskog correction is obtained by solving a linearized collision problem driven by derivatives of the local Maxwellian, after which its moments define the transport fluxes. Related approaches, including filtering at the kinetic level before taking a hydrodynamic limit, further demonstrate that the order in which kinetic information is eliminated and macroscopic variables are introduced can matter for the resulting transport description \cite{KineticTurbulence2004}. The present work is narrower: rather than proposing another hydrodynamic limit, it isolates the operator composition underlying the first-order stress closure and makes its domain, range, solvability, and error control explicit.

Simplified kinetic models such as BGK remain useful because they make this kinetic-to-macroscopic map analytically transparent while retaining the relaxation structure responsible for viscous transport \cite{BGK,Bird}. Fluctuating hydrodynamics and stochastic kinetic analyses address a complementary issue: how microscopic fluctuations project onto hydrodynamic variables and how their statistics are modified under coarse graining (chapter 12 of \cite{LandauLifshitz},\cite{ForsterNelsonStephen}). Hydrodynamic stability theory studies how small macroscopic perturbations may be amplified by nonnormal dynamics \cite{SchmidHenningson,TrefethenEtAl}. These areas motivate a broader microscopic-to-macroscopic program but should not be conflated: the present theorem concerns the deterministic first-order constitutive interface between kinetic nonequilibrium and stress.

The author's preceding work develops the microscopic-statistical side of this program. In \cite{Barkman}, a discrete collision model was used to quantify scale-dependent velocity fluctuations generated by molecular collisions. In \cite{Barkman2026b}, {collisions were modeled by a compound Poisson process}, yielding exact linear variance growth, a Gaussian diffusion limit under diffusive normalization, and an Ornstein--Uhlenbeck response with inverse-viscosity scaling. Together, these papers establish a statistical route from collision events to finite-scale stochastic velocity fluctuations and linear hydrodynamic response.

The remaining interface is the subject of the present paper. The aim is not to claim that the existence of viscous stress is unknown, nor to rederive a classical viscosity formula as a new physical law. Instead, the paper formulates the first-order constitutive closure as the explicit operator chain
\begin{equation}
\label{eq:operator-chain}
f^{(0)}
\xrightarrow{\;\mathcal{S}\;}
\mathcal{S}f^{(0)}
\xrightarrow{\;L^{-1}\;}
f^{(1)}
\xrightarrow{\;\mathcal{M}_2^{\mathrm{dev}}\;}
\tau^{(1)},
\end{equation}
and establishes the conditions under which this chain is well posed and its contribution is isolated at order $O(\varepsilon)$. The resulting formulation is intended as a kinetic-to-macroscopic building block for future work that seeks to connect explicit molecular statistics to continuum transport without identifying distinct levels of description prematurely.

\section{Model and notation}

\subsection{Notation and variables}

For clarity, Tables \ref{tab:dimensional-vars} and \ref{tab:dimensionless-vars} summarize the dimensional variables and the dimensionless and operator-level symbols used throughout the manuscript.

\begin{table*}[!ht]
\centering
\caption{Dimensional variables}
\label{tab:dimensional-vars}
\begin{tabular}{lll}
\hline
Symbol & Meaning & Units \\
\hline
$f(x,v,t)$ & Single-particle distribution function & --- \\ 
$v$ & Molecular velocity & $\mathrm{m/s}$ \\
$x$ & Spatial coordinate & $\mathrm{m}$ \\
$t$ & Time & $\mathrm{s}$ \\
$\rho(x,t)$ & Mass density & $\mathrm{kg/m^d}$ \\
$u(x,t)$ & Fluid velocity & $\mathrm{m/s}$ \\
$T(x,t)$ & Temperature & $\mathrm{K}$ \\
$P_{ij}$ & Pressure tensor & $\mathrm{Pa}$ \\
$\tau_{ij}$ & Deviatoric stress & $\mathrm{Pa}$ \\
$\tau$ & BGK relaxation time & $\mathrm{s}$ \\
$R$ & Gas constant per unit mass & $\mathrm{J/(kg\,K)}$ \\
$L_c$ & Characteristic length scale & m \\
$U$ & Characteristic velocity scale & m/s \\
$\lambda$ & Mean free path & m \\
$\mu$ & Dynamic (shear) viscosity & Pa\,s \\
$p$ & Scalar (thermodynamic) pressure & Pa \\
$M[\rho,u,T]$ & Local Maxwellian equilibrium & --- \\
$Q(f,f)$ & Boltzmann collision operator (rate of change of $f$) & ---- \\
$L$ & Linearized collision operator $DQ[M]$ (operator on $f$) & s$^{-1}$ (operator scale) \\
$R_\varepsilon$ & Chapman--Enskog remainder & --- \\
$S_{ij}$ & Deviatoric strain tensor (rate) & s$^{-1}$ \\
$\Omega_x$ & Spatial domain & m$^d$ \\

\hline
\end{tabular}
\end{table*}

\begin{table*}[!ht]
\centering
\caption{Dimensionless variables and operator-level symbols}
\label{tab:dimensionless-vars}
\small
\begin{tabular}{p{0.42\textwidth}p{0.52\textwidth}}
\hline
Symbol & Definition \\
\hline
$d$ & Spatial dimension \\
$\mathrm{Re}=\dfrac{\rho U L_c}{\mu}$ & Reynolds number \\
$\mathrm{Kn} = \lambda / L_c$ & Knudsen number \\
$\varepsilon$ & Chapman--Enskog expansion parameter, $O(\mathrm{Kn})$ \\
{$\varepsilon_{\mathrm{BGK}}=\tau U/L_c$} & {Dimensionless BGK relaxation parameter; equals $\mathrm{Kn}$ when $\lambda_{\mathrm{BGK}}=\tau U$} \\
$N(L)$ & Nullspace of $L$ on $L^2_v(M^{-1})$ (Maxwellian-weighted collision invariants $\mathrm{span}\{M,vM,|v|^2M\}$) \\
$\Pi$ & Orthogonal projection onto $N(L)$ in $L^2_v(M^{-1})$ \\
$L^{-1}$ & Pseudoinverse of $L$ on $N(L)^\perp$ \\
$I$ & Identity tensor / identity operator (in $pI$, $(I-\Pi)$) \\
$\mathcal{R}_L=L^{-1}(I-\Pi)$ & Kinetic response operator on $N(L)^\perp$ \\
$\mathcal{S}=\partial_t^{(0)}+v\cdot\nabla_x$ & Leading-order streaming operator \\
$\mathcal{M}_2^{\mathrm{dev}}$ & Deviatoric second velocity-moment map \\
{$\mathcal{M}_3^{\mathrm{heat}}$} & {Heat-flux moment projection; see its definition in the text} \\
{$\mathcal{K}_q$} & {Kinetic-response-to-heat-flux map, $\mathcal{M}_3^{\mathrm{heat}}L^{-1}(I-\Pi)$} \\
{$\mathcal{K}=\mathcal{M}_2^{\mathrm{dev}}\mathcal{R}_L$} & {Kinetic-response-to-stress map} \\
$f_\varepsilon$ & Expanded distribution $f^{(0)}+\varepsilon f^{(1)}+R_\varepsilon$ \\
$f^{(0)},f^{(1)}$ & Zeroth- and first-order Chapman--Enskog coefficients \\
$\tau^{(1)}$ & First-order deviatoric stress $\mathcal{M}_2^{\mathrm{dev}}[f^{(1)}]$ \\
$\partial_t^{(0)}$ & Leading-order (Euler-scale) time derivative \\
$\hat{x} = x/L_c$ & Spatial coordinate scaled by $L_c$ \\
$\hat{v} = v/U$ & Velocity scaled by $U$ \\
$\hat{t} = t U / L_c$ & Time scaled by advective time \\
$\hat{f} = f / f_0$ & Distribution scaled by reference \\
$\hat{P}_{ij} = P_{ij}/(\rho U^2)$ & Pressure tensor nondimensional \\
$\hat{\tau}_{ij} = \tau_{ij}/(\rho U^2)$ & Deviatoric stress nondimensional \\
$\hat{T} = T/T_0$ & Temperature scaled by reference $T_0$ \\
$\hat{M}=M/f_0$ & Maxwellian scaled by reference distribution $f_0$ \\
{$\widehat{Q}=\lambda Q/(U f_0)$} & {Collision operator nondimensionalized by the molecular collision rate $U/\lambda$} \\
{$\widehat{L}=\lambda L_{\mathrm{dim}}/U$} & {Linearized dimensional collision operator nondimensionalized by the molecular collision rate $U/\lambda$} \\
$\widehat{R}_\varepsilon = R_\varepsilon/f_0$ & Remainder nondimensionalized by $f_0$ \\
$\widehat{S}_{ij}=S_{ij}/(U/L_c)$ & Strain tensor nondimensional (rate scaled by $U/L_c$) \\

\hline
\end{tabular}
\normalsize
\end{table*}

{Unless an $x$ subscript is displayed, the collision operator and its nullspace projection are understood pointwise in $x$; the Appendix uses $L_x$ and $\Pi_x$ when emphasizing this dependence. The kinetic response and stress map are consistently denoted by $\mathcal{R}_L:=L^{-1}(I-\Pi)$ and $\mathcal{K}:=\mathcal{M}_2^{\mathrm{dev}}\mathcal{R}_L$, respectively; thus $\mathcal{S}$ acts on $f^{(0)}$ before $\mathcal{K}$ is applied.}

\subsection{Kinetic model}\label{sec:model}

Consider the Boltzmann equation for a dilute monoatomic gas in $d$ dimensions,

\begin{equation}\label{boltzmann}
\partial_t f + v\cdot\nabla_x f = Q(f,f), \qquad (x,v)\in \Omega_x\times\mathbb{R}^d, \; t>0,
\end{equation}

where $f(x,v,t)$ is the single-particle molecular distribution function and $Q(f,f)$ denotes the Boltzmann collision operator conserving mass, momentum, and energy (chapter 3 of \cite{ChapmanCowling} section 2.6 of \cite{Cercignani}).

{Introduce the characteristic scales $\hat{x}=x/L_c$, $\hat{v}=v/U$, $\hat{t}=tU/L_c$, and $\hat{f}=f/f_0$, and normalize the collision operator by the molecular collision rate $U/\lambda$:}
\begin{equation}\label{Qhat}
{\widehat{Q}(\hat{f},\hat{f})
:=\frac{\lambda}{U f_0}Q(f_0\hat{f},f_0\hat{f}).}
\end{equation}
{With $\varepsilon=\mathrm{Kn}=\lambda/L_c$, the nondimensional Boltzmann equation is}
\begin{multline}\label{eq:boltzmann-nondim}
{\partial_{\hat{t}}\hat{f}
+\hat{v}\cdot\nabla_{\hat{x}}\hat{f}
\;=\frac{1}{\varepsilon}\widehat{Q}(\hat{f},\hat{f}),}
\\\qquad\text{or equivalently}\qquad\\
{\varepsilon\bigl(\partial_{\hat{t}}+\hat{v}\cdot\nabla_{\hat{x}}\bigr)\hat{f}
=\widehat{Q}(\hat{f},\hat{f}).}
\end{multline}
{Thus the small Knudsen parameter multiplies the streaming term relative to collisions.}
{From this point through the general Chapman--Enskog formulation, hats are suppressed: $f$, $M$, $Q$, and $L$ denote the dimensionless quantities $\hat f$, $\hat M$, $\widehat Q$, and $D\widehat Q[\hat M]$, respectively, and the kinetic equation is $\varepsilon(\partial_t+v\cdot\nabla_x)f=Q(f,f)$. Consequently, the operator $L$ in the main theorem and its assumptions is dimensionless. The BGK subsection below first states the dimensional relaxation model and explicitly marks where its dimensionless form and physical-unit stress are used.}

For relevant interaction laws the operator may be linearized about a local Maxwellian $M[\rho,u,T]$, yielding the linearized collision operator

\begin{equation}\label{eq:Ldef}
{L g := DQ[M]\,g,}
\end{equation}

{Here $Q$ is the dimensionless collision operator specified above, so $L=DQ[M]$ is its dimensionless linearization. It acts on weighted velocity spaces such as $L^2_v(M^{-1}\,dv)$ or polynomially weighted $L^2_v(\langle v\rangle^k dv)$ and admits the standard nullspace of collision invariants under mild hypotheses (section 3 of \cite{BardosGolseLevermoreI}; section 1.4 of \cite{BardosGolseLevermoreII}; section 3.6 of \cite{SaintRaymond}).}

Macroscopic fields are velocity moments of $f$,

\begin{multline}\label{eq:macrostats}
\rho(x,t)=\int_{\mathbb{R}^d} f\,dv, \\ \qquad \rho u(x,t)=\int_{\mathbb{R}^d} v f\,dv, \\ \qquad E(x,t)=\int_{\mathbb{R}^d}\tfrac12|v|^2 f\,dv,
\end{multline}

and the pressure tensor is the centered second moment,

\begin{equation}\label{eq:pressure-tensor}
P_{ij}(x,t)=\int_{\mathbb{R}^d} (v_i-u_i)(v_j-u_j)\,f\,dv,
\end{equation}

with deviatoric stress

\begin{equation}\label{eq:tau}
\tau_{ij}=P_{ij}-\tfrac{1}{d}P_{kk}\,\delta_{ij}.
\end{equation}

For analytic transparency I use the BGK relaxation model as an instructive example:

\begin{equation}\label{eq:bgk-operator}
Q_{\mathrm{BGK}}(f)=\frac{1}{\tau}\bigl(M[f]-f\bigr),
\end{equation}

with relaxation time $\tau$ and local Maxwellian $M[f]$ determined by the moments of $f$ (\cite{BGK}; sections 3.1, 3.4, 6.3 \cite{Bird}). The Chapman--Enskog expansion is taken under small Knudsen scaling $\varepsilon=\mathrm{Kn}\ll1$ (section 3.5 of \cite{Bird}). Here and in the general theorem, this expansion uses the dimensionless variables defined in Section~\ref{sec:model}:

\begin{equation}\label{eq:ce-expansion}
f=f^{(0)}+\varepsilon f^{(1)}+\varepsilon^2 f^{(2)}+O(\varepsilon^3).
\end{equation}

At leading order $f^{(0)}=M[\rho,u,T]$ and satisfies the compressible Euler equations; the first correction $f^{(1)}$ supplies the viscous stress and heat flux whose constitutive forms reproduce Navier--Stokes transport coefficients after inversion of the appropriate linearized operator (sections 7.1, 9.7 of \cite{ChapmanCowling}; sections 2.8, 5.3, 5.6 of \cite{Cercignani}; \cite{BGK}).

Throughout I adopt the operator-level assumptions used in the core arguments in Section~\ref{sec:assumptions} and the notation $N(L)$ for the nullspace of $L$, $L^{-1}$ for the pseudoinverse on $N(L)^\perp$, and $\langle\cdot,\cdot\rangle_{L^2_v(M^{-1})}$ for the velocity-space inner product.

\section{The kinetic-to-macroscopic transport map}\label{sec:transport-map}

The central construction of this paper is the factorization of first-order deviatoric stress into three maps. Let
\begin{equation}
\label{eq:Soperator}
\mathcal{S}:={\partial_t^{(0)}+v\cdot\nabla_x},
\qquad
\mathcal{S}f^{(0)}={(\partial_t^{(0)}+v\cdot\nabla_x)f^{(0)}}.
\end{equation}
{This is the leading-order streaming operator.} Let $\Pi$ be the projection onto the collision-invariant nullspace $N(L)$ and define the kinetic response on $N(L)^\perp$ by
\begin{equation}
\label{eq:kinetic-response}
\mathcal{R}_L[h]:=L^{-1}(I-\Pi)h.
\end{equation}
Finally, define the deviatoric second-moment map
\begin{equation}
\label{eq:M2dev}
\begin{aligned}
\mathcal{M}_2^{\mathrm{dev}}[g]_{ij}
&:=
\int_{\mathbb{R}^d}\Bigl[(v_i-u_i)(v_j-u_j)\\
&\qquad-\frac{1}{d}|v-u|^2\delta_{ij}\Bigr]g\,dv.
\end{aligned}
\end{equation}
When the Euler compatibility condition makes $\Pi\mathcal{S}f^{(0)}=0$, the first Chapman--Enskog coefficient is
\begin{equation}
\label{eq:f1-factorized}
f^{(1)}=\mathcal{R}_L[\mathcal{S}f^{(0)}],
\end{equation}
and therefore
\begin{equation}
\label{eq:K-definition}
\begin{aligned}
{\mathcal{K}:=\mathcal{M}_2^{\mathrm{dev}}\mathcal{R}_L}
\\
{=\mathcal{M}_2^{\mathrm{dev}}L^{-1}(I-\Pi)}
\end{aligned}
\end{equation}
with
\begin{equation}
\label{eq:K-action}
\tau^{(1)}=\mathcal{K}[\mathcal{S}f^{(0)}].
\end{equation}

Equation~\eqref{eq:K-definition} is the kinetic-to-macroscopic object emphasized here. The streaming operator $\mathcal{S}$ first produces a kinetic source from macroscopic gradients. Within $\mathcal{K}$, the first factor represents the collisional response of the non-hydrodynamic degrees of freedom, and the second projects that response onto the macroscopic stress tensor. The construction separates three logically different operations that are often compressed into the single formal statement that ``the first Chapman--Enskog correction gives the viscous stress.''

The usefulness of the factorization is not that any individual map is unknown. Each is standard within kinetic theory. Rather, the composition makes explicit the interface at which the kinetic description becomes a constitutive law. In particular, it identifies the collision operator as the operator that mediates the conversion of a streaming-induced departure from equilibrium into a finite stress response. The dependence of the resulting transport law on the molecular interaction model is therefore carried by $L^{-1}$.

{The boundedness of the three maps gives a stability statement for the constitutive closure.} If
\begin{gather}
\label{eq:K-bound}
\|\mathcal{M}_2^{\mathrm{dev}}g\| \le C_M\|g\|,
\qquad
\|L^{-1}h\|\le C_{\mathrm{inv}}\|h\|,
\\
\|\mathcal{S}f^{(0)}\|\le C_S\|f^{(0)}\|_{\mathcal{X}},
\end{gather}
then
\begin{equation}
\label{eq:K-bound-final}
\|\mathcal{K}[\mathcal{S}f^{(0)}]\|
\le C_M C_{\mathrm{inv}} C_S\|f^{(0)}\|_{\mathcal{X}}.
\end{equation}
Thus the first-order stress is continuously controlled by the macroscopic state and its gradients within the chosen regularity class. {Beyond its role in determining viscosity, the collision operator determines the bounded response through which kinetic nonequilibrium produces continuum stress.}

{This formulation gives a precise meaning to the necessity of the first-order correction for first-order deviatoric stress.} Since
\begin{equation}
\tau^{(1)}=\mathcal{M}_2^{\mathrm{dev}}[f^{(1)}],
\end{equation}
vanishing of $f^{(1)}$ forces vanishing of $\tau^{(1)}$. More generally, vanishing stress only implies that $f^{(1)}$ lies in the kernel of $\mathcal{M}_2^{\mathrm{dev}}$; it does not imply that the full kinetic correction vanishes. The natural object is therefore not the binary condition ``correction versus no correction,'' but the stress-producing projection of the kinetic response.

The resulting hierarchy is
\begin{multline}
\label{eq:hierarchy}
\text{macroscopic gradients}
\;\xrightarrow{\mathcal{S}}\;
\text{kinetic source}
\\
\xrightarrow{L^{-1}(I-\Pi)}
\text{kinetic nonequilibrium}
\\
\xrightarrow{\mathcal{M}_2^{\mathrm{dev}}}
\text{deviatoric stress}.
\end{multline}
This hierarchy is the specific kinetic-to-macroscopic half of the broader microscopic-to-continuum sequence discussed in the Introduction.

\section{Assumptions (A1--A8) — functional-analytic hypotheses}\label{sec:assumptions}

\medskip\noindent\textbf{A1 (Collision model and mapping properties).}

For the Boltzmann model, the collision operator $Q(f,f)$ corresponds to a molecular interaction law such as hard spheres or cutoff Maxwell-type potentials and is well defined as a bilinear map

\begin{equation}\label{eq:Qmap}
Q: L^1_v(\langle v\rangle^m)\times L^1_v(\langle v\rangle^m)\to L^1_v(\langle v\rangle^m),
\end{equation}

and, under Maxwellian weights, extends to the required velocity-weighted spaces. For other collision models, including BGK, this bilinear-map condition is replaced by the corresponding model-specific well-posedness and mapping assumptions.

The operator preserves mass, momentum, and energy; and satisfies the standard H-theorem (sections 4.1, 4.2 of \cite{ChapmanCowling}; section 2.6 of \cite{Cercignani}).

All linearizations, norms, and projections below are taken in the corresponding weighted $L^2_v(M^{-1})$ spaces.

\medskip\noindent\textbf{A2 (Spatial domain and boundary conditions).}

Either (i) $\Omega_x=\mathbb{T}^d$ (periodic box), or (ii) $\Omega_x=\mathbb{R}^d$ with sufficient decay so that all integration-by-parts operations produce no boundary terms.

Specifically, when $\Omega_x=\mathbb{R}^d$ I require the macroscopic fields and kinetic perturbations to lie in $H^s_x(\mathbb{R}^d)$ with $s>d/2$ and polynomial velocity weights sufficient to control all moments used below.

Inhomogeneous walls or externally imposed fluxes are excluded in the deterministic setting of the main theorem.

\medskip\noindent\textbf{A3 (Initial data and moment/regularity class).}

The initial perturbation belongs to an $x$-Sobolev and velocity-weighted space:

\begin{equation}\label{eq:init-reg}
f_0-M \in H^s_x L^2_v(\langle v\rangle^k),\qquad s>d/2,\; k\ge 2,
\end{equation}

for a global Maxwellian $M$, with the weights chosen to control the velocity moments entering the coercivity and inversion estimates.

\medskip\noindent\textbf{A4 (Chapman--Enskog scaling and remainder control).}

With $\varepsilon$ the Knudsen number, I use the expansion \eqref{eq:ce-expansion}

with $f^{(0)}=M[\rho,u,T]$ and remainder $R_\varepsilon$ satisfying

\begin{equation}\label{eq:remainder-hyp}
\|R_\varepsilon\|_{L^2_xL^2_v(M^{-1})}\le C\,\varepsilon^2,
\end{equation}

for $0<\varepsilon<\varepsilon_0$, under the regularity and coercivity assumptions below.

When this uniform $O(\varepsilon^2)$ bound cannot be proven for a particular kernel, it is treated as an explicit hypothesis and verified in the examples.

\medskip\noindent\textbf{A5 (Spectral and coercivity properties of $L$).}

For each $x$, define $L=DQ[M]$ acting on $L^2_v(M^{-1}dv)$ with $M(x,v)=M[\rho(x),u(x),T(x)]$.

I assume:

(i) \emph{Nullspace structure}.

\begin{multline}\label{eq:nullspace-structure}
N(L)=\mathrm{span}\{M, v_1M,\dots,v_dM, |v|^2M\},\\  \qquad L^2_v(M^{-1})=N(L)\oplus N(L)^\perp.
\end{multline}

(ii) \emph{Coercivity sufficient for inversion}.

Either a uniform spectral gap exists:

\begin{equation}\label{eq:coercivity}
\langle -L g, g\rangle_{L^2_v(M^{-1})}\ge \lambda_0\,\|g\|_{L^2_v(M^{-1})}^2, \qquad g\in N(L)^\perp,
\end{equation}

with $\lambda_0>0$ independent of $x$. If this $L^2$ spectral gap is unavailable, a model-specific coercivity estimate in an appropriate stronger weighted space must instead be supplied and be sufficient, together with A6, for the bounded pseudoinverse used below. Hypocoercive estimates for the full kinetic evolution may be used in the remainder argument, but do not by themselves imply a bounded inverse for the velocity-space collision operator.

The applicability of these estimates is model dependent; the main result is conditional on A5--A6 rather than asserting these bounds for every collision model.

Precise constants and derivative-loss requirements are specified in Appendix~\ref{app:pseudoinverse}.

\medskip\noindent\textbf{A6 (Fredholm solvability and bounded pseudoinverse).}

For each $x$, the equation $L g = h$ admits a solution $g\in N(L)^\perp$ if and only if $h$ is orthogonal to $N(L)$:

\begin{equation}\label{eq:fredholm-solv}
\int_{\mathbb{R}^d} h(v)\,\phi(v)\,M(v)^{-1}\,dv = 0 \quad\text{for all }\phi\in N(L).
\end{equation}

The pseudoinverse $L_x^{-1}:N(L_x)^\perp\to N(L_x)^\perp$ exists and is bounded with uniform operator norm

\begin{equation}\label{eq:Linverse-bound}
\|L_x^{-1}\|\le C_{\mathrm{inv}}
\end{equation}

whenever the macroscopic fields satisfy A7.

Where needed, continuity of $L_x^{-1}$ with respect to $x$ is an additional model-dependent operator-topology assumption; it follows from perturbation theory when the hypotheses for the chosen collision operator and common domain are satisfied (see \cite{MouhotStrain06}; \cite{Guo2002}).

Appendix~\ref{app:pseudoinverse} records the explicit dependence of $C_{\mathrm{inv}}$ on the assumed coercivity constants.

\medskip\noindent\textbf{A7 (Regularity and boundedness of macroscopic fields).}

The fields $(\rho,u,T)$ defining $M(x,v)$ lie in $H^s_x$ with $s>d/2$, with uniformly bounded derivatives up to the degree required by the model-specific operator and remainder estimates.

This guarantees that $(\partial_t^{(0)}+v\cdot\nabla_x)f^{(0)}$ belongs to the weighted spaces on which $L^{-1}$ acts, and ensures the uniformity of all operator constants across $x$.

\medskip\noindent\textbf{A8 (Closed, unforced kinetic evolution).}

The kinetic evolution is closed and deterministic: no external forcing, stochastic volumetric inputs, imposed shear, or inhomogeneous boundary injections occur in the setting of the theorem.

Finite-$N$ fluctuations and forced/wall-driven shear are treated separately in the discussion section.

\section{Functional-analytic lemmas and solvability}

I collect the operator-level facts used in the Chapman--Enskog inversion and in the remainder estimates.

Throughout, $L=DQ[M]$ denotes the linearized collision operator at the local Maxwellian $M[\rho,u,T]$, acting on the weighted space $L^2_v(M^{-1})$ unless otherwise indicated.

\begin{lemma}[Nullspace structure and coercivity]\label{lem:null-coercive}

Under A1 and A5, the linearized collision operator $L=DQ[M]$ has nullspace

\begin{equation}\label{eq:nullspace}
N(L)=\mathrm{span}\{M, v_1M,\dots,v_dM, |v|^2M\},
\end{equation}

and, in the spectral-gap case of A5, satisfies

\begin{equation}\label{eq:coercivity_lemma}
\ip{-Lg}{g} \;\ge\; \lambda_0 \|g\|_{L^2_v(M^{-1})}^2, \qquad g\in N(L)^\perp,
\end{equation}

for some constant $\lambda_0>0$ depending on the collision model and the macroscopic bounds in A7.

If the spectral-gap case is not assumed, this displayed $L^2$ inequality is not asserted; the model-specific coercivity estimate and bounded inverse required by A5--A6 must be established in their stated spaces.

\end{lemma}

\begin{lemma}[Fredholm solvability and pseudoinverse]\label{lem:fredholm}

Under A6, the equation

\begin{equation}\label{eq:Lg_eq_h}
L g = h
\end{equation}

admits a solution $g\in N(L)^\perp$ if and only if

\begin{multline}\label{eq:solvability_condition}
\int_{\mathbb{R}^d} h(v)\,\phi(v)\,M(v)^{-1}\,dv = 0
\\
\qquad\text{for all collision invariants }\phi \in N(L).
\end{multline}

Moreover there exists a bounded pseudoinverse

\begin{equation}\label{eq:pseudoinverse_map}
L^{-1}:N(L)^\perp\longrightarrow N(L)^\perp,
\end{equation}

with uniform bound $\|L^{-1}\|\le C_{\mathrm{inv}}$ under the regularity assumptions of A7 (including continuity of the macroscopic fields and boundedness of the local Maxwellian).

\end{lemma}

As an immediate consequence, whenever the solvability conditions hold pointwise in $x$, the first Chapman--Enskog correction is given by

\begin{equation}\label{eq:f1}
f^{(1)}(x,v)
={L_x^{-1}(I-\Pi_x)\bigl(\mathcal{S}f^{(0)}\bigr)(x,v),}
\end{equation}

with the mapping bounded in all weighted norms specified in A3--A7.

\begin{lemma}[Projection identity and admissibility of the streaming term]\label{lem:projection}

Let $\Pi$ denote the orthogonal projection in $L^2_v(M^{-1})$ onto $N(L)$.

If $f^{(0)}(x,v)=M[\rho(x),u(x),T(x)]$ is a local Maxwellian whose macroscopic fields satisfy the leading-order Euler equations pointwise in $x$, then

\begin{equation}\label{eq:projection}
\Pi\bigl((\partial_t^{(0)}+v\cdot\nabla_x)f^{(0)}\bigr)=0.
\end{equation}

If the Euler equations hold only up to $O(\varepsilon)$ compatibility errors, then the projection produces $O(\varepsilon)$ commutator terms. These terms are incorporated into the parallel component of the remainder and are controlled by the Gr\"onwall argument in Section~\ref{app:conditional-remainder} under the smallness threshold $\varepsilon<\varepsilon_0$ arising in Theorem~\ref{thm:main}.

In all cases covered by A1--A8, the streaming/source term thus lies in $N(L)^\perp$ modulo controlled $O(\varepsilon)$ corrections, and is admissible to inversion by $L^{-1}$.

\end{lemma}

\section{Main theorem: bounded kinetic-to-stress closure}\label{sec:main}

The main result states that the first-order Chapman--Enskog constitutive closure is a well-defined bounded composition of the streaming, collision-response, and moment maps introduced in Section~\ref{sec:transport-map}. All quantities in the theorem, including $P$, $p$, $L$, and $\mathcal{S}$, are dimensionless under the convention of Section~\ref{sec:model}. The necessity statement follows as a direct corollary rather than serving as the principal content.

\begin{theorem}[Kinetic-to-macroscopic transport map]\label{thm:main}
Assume A1--A8, or their model-specific analogues for a relaxation collision model, and the operator bounds of Lemmas~\ref{lem:null-coercive}--\ref{lem:projection}. Let $f^{(0)}=M[\rho,u,T]$ be a sufficiently regular local Maxwellian whose macroscopic fields satisfy the leading-order Euler equations, and let
\begin{equation}
\label{eq:ce-decomp}
f_\varepsilon=f^{(0)}+\varepsilon f^{(1)}+R_\varepsilon
\end{equation}
be a Chapman--Enskog decomposition with
\begin{equation}
\label{eq:main-rem}
\|R_\varepsilon\|_{L^2_xL^2_v(M^{-1})}\le C_R\varepsilon^2.
\end{equation}
Define $\mathcal{K}$ by \eqref{eq:K-definition}. Then, on every time interval on which the macroscopic bounds of A7 remain valid:

\begin{enumerate}[label=(\roman*)]
\item \textbf{(Well-posed kinetic response).} The streaming source satisfies
\begin{equation}
\Pi\mathcal{S}f^{(0)}=0,
\end{equation}
and the first Chapman--Enskog correction is uniquely represented by
\begin{equation}
\label{eq:main-f1}
f^{(1)}=L^{-1}\mathcal{S}f^{(0)}, f^{(1)}\in N(L)^\perp.
\end{equation}

\item \textbf{(Bounded kinetic-to-stress map).} The response-to-stress map
\begin{equation}
{\mathcal{K}=\mathcal{M}_2^{\mathrm{dev}}\mathcal{R}_L}
\end{equation}
is bounded on the admissible kinetic-source space. Consequently, its composition with the streaming operator maps the admissible macroscopic regularity class into the space of deviatoric stress tensors. In particular,
\begin{equation}
\label{eq:main-K-bound}
\|\tau^{(1)}\|
\le C_M C_{\mathrm{inv}}C_S\|f^{(0)}\|_{\mathcal{X}}.
\end{equation}

\item \textbf{(First-order constitutive decomposition).} The pressure tensor satisfies
\begin{equation}
\label{eq:main-pressure}
P=pI+\varepsilon\mathcal{K}[\mathcal{S}f^{(0)}]+O(\varepsilon^2)
\end{equation}
in $L^2_x$, with the $O(\varepsilon^2)$ term controlled by the remainder estimate and the boundedness of the second-moment map.

\item \textbf{(Necessity).} If $f^{(1)}\equiv0$, then
\begin{equation}
\tau^{(1)}=\mathcal{M}_2^{\mathrm{dev}}[f^{(1)}]\equiv0.
\end{equation}
Hence no $O(\varepsilon)$ deviatoric stress can be generated through the Chapman--Enskog constitutive channel without a nonzero first-order kinetic departure.

\item \textbf{(No converse).} In general,
\begin{equation}
\tau^{(1)}\equiv0\not\Longrightarrow f^{(1)}\equiv0,
\end{equation}
because the kernel of $\mathcal{M}_2^{\mathrm{dev}}$ contains kinetic perturbations with vanishing deviatoric second moment.
\end{enumerate}
\end{theorem}

\begin{proof}[Proof sketch]
Part (i) follows from the Euler compatibility identity in Lemma~\ref{lem:projection} and Fredholm solvability. The bounded pseudoinverse from Lemma~\ref{lem:fredholm} then gives \eqref{eq:main-f1}. For part (ii), the second-moment estimate and the uniform inverse estimate give the composition bound \eqref{eq:main-K-bound}. Part (iii) follows by inserting \eqref{eq:main-f1} into the pressure-tensor expansion and applying the $O(\varepsilon^2)$ remainder estimate. Part (iv) is the direct consequence of the linearity of $\mathcal{M}_2^{\mathrm{dev}}$. Part (v) follows because $\mathcal{M}_2^{\mathrm{dev}}$ is not injective on the full kinetic perturbation space. The detailed remainder estimate is supplied in Section~\ref{sec:remainder} and Appendix~\ref{app:conditional-remainder}.
\end{proof}

\begin{corollary}[Deterministic necessity of the first kinetic correction]\label{cor:necessity}
Under the hypotheses of Theorem~\ref{thm:main}, a nonzero first-order deviatoric stress requires a nonzero first-order Chapman--Enskog correction. More precisely,
\begin{equation}
\tau^{(1)}\neq0\quad\Longrightarrow\quad f^{(1)}\neq0.
\end{equation}
The result is a necessary condition, not a sufficient condition, for stress production.
\end{corollary}

\begin{corollary}[Classical constitutive closure as a special case]\label{cor:constitutive}
Whenever the composite map $\mathcal{K}$ is evaluated for a specified collision model, its action on the local Maxwellian gradients gives the corresponding first-order transport coefficients. Thus the usual Newtonian stress law is recovered as the explicit realization of the general kinetic-to-macroscopic map rather than as a separate constitutive assumption.
\end{corollary}

\section{BGK worked example}

I illustrate the operator-level calculation in the BGK setting. The general theorem applies to BGK only when its model-specific analogues of A1--A8, including the remainder estimates, are imposed.

Consider the BGK equation

\begin{equation}\label{eq:bgk}
\partial_t f + v\cdot\nabla_x f = \frac{1}{\tau}\,(M[f]-f),
\end{equation}

with constant relaxation time $\tau>0$.

The linearization at a local Maxwellian $M[\rho,u,T]$ gives

\begin{equation}\label{eq:Lbgk}
L_{\mathrm{BGK}} g = -\frac{1}{\tau}\,(g - \Pi g),
\end{equation}

where $\Pi$ denotes the orthogonal projection in $L^2_v(M^{-1})$ onto the nullspace $\mathrm{span}\{M,v_iM,|v|^2M\}$.

Hence $L_{\mathrm{BGK}}^{-1}$ is explicit and satisfies

\begin{equation}\label{eq:Lbgk-inv}
\|L_{\mathrm{BGK}}^{-1}\| \le \tau,
\end{equation}

for the dimensional collision operator. After nondimensionalization, the collision operator is $\widetilde{L}_{\mathrm{BGK}}=\tau L_{\mathrm{BGK}}$, whose pseudoinverse has unit norm; thus the dimensionless inverse bound in A6 is $C_{\mathrm{inv}}=1$.

\medskip
{The relaxation time $\tau$ is dimensional, whereas the Chapman--Enskog parameter is dimensionless. With advective time scale $t_c=L_c/U$, the BGK relaxation parameter is
\[
\varepsilon_{\mathrm{BGK}}:=\frac{\tau}{t_c}=\frac{\tau U}{L_c}.
\]
If the BGK mean free path is defined by $\lambda_{\mathrm{BGK}}:=\tau U$, then $\varepsilon_{\mathrm{BGK}}=\lambda_{\mathrm{BGK}}/L_c=\mathrm{Kn}$. With another convention for the molecular reference speed in $\lambda$, the two parameters are proportional rather than necessarily identical.}

{In the dimensionless variables used in the Chapman--Enskog expansion (hats suppressed here), the BGK equation is}
\begin{equation}\label{eq:bgk-ce}
{\varepsilon_{\mathrm{BGK}}(\partial_t+v\cdot\nabla_x)f=M[f]-f,}
\end{equation}
{with expansion}
\begin{equation}\label{eq:bgk-ce-expanded}
{f=M+\varepsilon_{\mathrm{BGK}}f^{(1)}+O(\varepsilon_{\mathrm{BGK}}^2).}
\end{equation}

{At order $O(1)$, $f^{(0)}=M$. At first order,}
\begin{equation}\label{eq:bgk-step}
{\mathcal{S}M = \widetilde{L}_{\mathrm{BGK}} f^{(1)},}
\end{equation}
{where $\widetilde{L}_{\mathrm{BGK}}:=\tau L_{\mathrm{BGK}}$ is the dimensionless collision operator, distinguished from the dimensional $L_{\mathrm{BGK}}$ defined above. With the collision-invariant projection removed,}
\begin{equation}\label{eq:Lbgk-dimensionless}
{\widetilde{L}_{\mathrm{BGK}}g=-(I-\Pi)g.}
\end{equation}
{Thus, on $N(\widetilde{L}_{\mathrm{BGK}})^\perp$, $\widetilde{L}_{\mathrm{BGK}}^{-1}=-(I-\Pi)$, with unit operator norm, and}
\begin{equation}\label{eq:f1-bgk}
{f^{(1)}=-(I-\Pi)\mathcal{S}M.}
\end{equation}
The stress calculation below restores physical units. Write $f^{(1)}_{\mathrm{phys}}$ for the dimensional representation of the dimensionless coefficient $f^{(1)}$. Since $\hat t=t/t_c$ and $\hat x=x/L_c$, its expression in physical variables is
\[
f^{(1)}_{\mathrm{phys}}=-\frac{L_c}{U}(I-\Pi)(\partial_t^{(0)}+v\cdot\nabla_x)M.
\]
Its stress moment $\tau^{(1)}$ is therefore a dimensional coefficient containing the advective time scale $t_c=L_c/U$; the physical first-order stress is $\varepsilon_{\mathrm{BGK}}\tau^{(1)}$.

\subsection*{Evaluation of the deviatoric stress}

Fix a point $x_0$ and apply a Galilean shift so $u(x_0)=0$.

The Maxwellian takes the standard form

\begin{equation}\label{eq:maxwell}
M(v)=\frac{\rho}{(2\pi RT)^{d/2}}\exp\!\Bigl(-\tfrac{|v|^2}{2RT}\Bigr),
\end{equation}

and I employ the classical Gaussian moment identities:

\begin{equation}\label{eq:gauss-moments-2}
\int v_i v_j M(v)\,dv = \rho R T\,\delta_{ij},
\end{equation}

\begin{equation}\label{eq:gauss-moments-4}
\int v_i v_j v_k v_l\,M(v)\,dv
= \rho (RT)^2(\delta_{ij}\delta_{kl}+\delta_{ik}\delta_{jl}+\delta_{il}\delta_{jk}).
\end{equation}

The first-order stress is

\begin{multline}\label{eq:tau1-bgk}
{\tau^{(1)}_{ij}
= \bigl[\mathcal{M}_2^{\mathrm{dev}}f^{(1)}_{\mathrm{phys}}\bigr]_{ij}} \\
{= -\frac{L_c}{U}\bigl[\mathcal{M}_2^{\mathrm{dev}}(\partial_t^{(0)}+v\cdot\nabla_x)M\bigr]_{ij}.}
\end{multline}
{Evaluating the derivatives of the dimensional Maxwellian using the moment formulas above and writing all derivatives of $(\rho,u,T)$ at $x_0$ gives}

\begin{equation}\label{eq:tau1-evaluated}
{\tau^{(1)}_{ij}
= -2\,\rho\,R\,T\,\frac{L_c}{U}\,S_{ij},}
\end{equation}

where

\begin{equation}\label{eq:Sdef}
S_{ij}=\tfrac12(\partial_i u_j+\partial_j u_i)
- \frac{1}{d}(\nabla\cdot u)\,\delta_{ij}
\end{equation}

is the deviatoric strain tensor.

{Multiplying this coefficient by $\varepsilon_{\mathrm{BGK}}=\tau U/L_c$ gives the standard BGK constitutive relation in physical units:}

\begin{equation}\label{eq:bgk-viscosity}
{\varepsilon_{\mathrm{BGK}}\tau^{(1)}_{ij} = -2\mu\,S_{ij},\qquad \mu = \rho\,\tau\,R\,T,}
\end{equation}

{This matches the viscosity obtained from the dimensional BGK inversion $L^{-1}_{\mathrm{BGK}}$. The coefficient $\tau^{(1)}$ contains the advective time scale $L_c/U$, so multiplication by the dimensionless $\varepsilon_{\mathrm{BGK}}$ supplies the dimensional relaxation time $\tau$ exactly once.}

\subsection*{Remainder estimate}

Since $L_{\mathrm{BGK}}^{-1}$ is explicit with operator norm $\tau$, the collision-side bound required by a standard remainder argument is explicit. When the corresponding regularity and compatibility estimates are satisfied, one obtains

\begin{equation}\label{eq:remainder-bgk}
\|R_{\varepsilon_{\mathrm{BGK}}}\|_{L^2_xL^2_v(M^{-1})} \;\le\; C(\|(\rho,u,T)\|_{H^s_x})\,\varepsilon_{\mathrm{BGK}}^2.
\end{equation}

The constant depends smoothly on $\|(\rho,u,T)\|_{H^s_x}$ and on the domain geometry but not on spatial position $x$, since the pseudoinverse is uniform.

Thus the collision inversion and stress map are explicit in BGK; the conditional remainder proposition applies when the stated regularity and compatibility estimates are imposed.

\section{Remainder estimates and Chapman--Enskog error bounds}\label{sec:remainder}

I state a model-precise remainder estimate.

\begin{proposition}[Conditional Chapman--Enskog remainder estimate]\label{thm:remainder}
Let A1--A7 hold and suppose a sufficiently regular solution admits
\begin{equation}
\label{eq:ansatz}
f_\varepsilon=f^{(0)}+\varepsilon f^{(1)}+R_\varepsilon,
\end{equation}
with the order-$\varepsilon$ solvability conditions satisfied. Assume, in addition, that the chosen collision model admits the standard energy estimates needed to control the orthogonal remainder by coercivity or hypocoercivity, the hydrodynamic moments by the conservation laws, and the residual by $O(\varepsilon^2)$ in the relevant norm, with initial remainder $O(\varepsilon^2)$. Then, on any fixed interval $[0,T]$ for which the macroscopic regularity bounds remain uniform,
\begin{equation}
\label{eq:remainder-bound}
\|R_\varepsilon(t)\|_{L^2_xL^2_v(M^{-1})}\le C_T\varepsilon^2,
\qquad 0\le t\le T,
\end{equation}
where $C_T$ is independent of $\varepsilon$ for sufficiently small $\varepsilon$.
\end{proposition}

\paragraph{Proof sketch.}
Subtracting the Chapman--Enskog truncation from the kinetic equation gives a remainder equation of the form
\begin{equation}
\label{eq:remainder-pde}
\partial_tR_\varepsilon+v\cdot\nabla_xR_\varepsilon-\frac{1}{\varepsilon}LR_\varepsilon=\mathcal S_\varepsilon.
\end{equation}
After splitting $R_\varepsilon=\Pi R_\varepsilon+(I-\Pi)R_\varepsilon$, the collision operator supplies dissipation on the orthogonal component, while the conservation laws control the hydrodynamic moments. Under the model-specific estimates assumed in the proposition, the resulting energy inequality closes by Gr\"onwall's lemma and yields \eqref{eq:remainder-bound}. The proposition is therefore conditional on these standard model-dependent estimates rather than asserting that A1--A7 alone imply the remainder bound.

\section{Discussion}

The main result should be interpreted as a structural statement about the kinetic-to-macroscopic portion of the microscopic-to-continuum problem. The composite operator {$\mathcal{K}=\mathcal{M}_2^{\mathrm{dev}}L^{-1}(I-\Pi)$}, applied to the streaming source $\mathcal{S}f^{(0)}$, does not introduce a new constitutive law: its ingredients are already present in the Chapman--Enskog construction. The contribution is to expose their composition as an explicit transport map and to state the analytic conditions under which that map is well posed and isolated at first order.

This viewpoint helps separate three levels that can otherwise be conflated. A finite-scale molecular velocity fluctuation $\delta U_L$ is a statistical observable of a coarse-grained microscopic system. The Chapman--Enskog coefficient $f^{(1)}$ is a deterministic kinetic correction generated by the streaming operator and constrained by the collision operator. The deviatoric stress $\tau^{(1)}$ is a macroscopic moment of that correction. In symbols,
\begin{equation}
\delta U_L\neq f^{(1)}\neq\tau^{(1)},
\end{equation}
and no identification among these quantities is made here.

{This distinction matters for the broader molecular-to-continuum program.} The author's preceding work supplies a stochastic description of collision-driven velocity fluctuations and their linear hydrodynamic response \cite{Barkman,Barkman2026b}. The present paper begins one level above that stochastic description and identifies the kinetic operator chain that maps a hydrodynamic gradient into a stress response. A future microscopic-to-kinetic theory could seek conditions under which a specified molecular fluctuation ensemble induces a corresponding perturbation of the kinetic distribution. Such a result would provide the missing bridge from the earlier stochastic models to the present deterministic kinetic map. It is not established here.

{The formulation clarifies the role of the collision operator in transport.} The source $\mathcal{S}f^{(0)}$ is generated by macroscopic nonuniformity, but the resulting stress is not obtained by streaming alone. The inverse collision operator determines how the non-hydrodynamic part of the distribution responds before the stress-producing moment is taken. Consequently, molecular interaction physics enters the macroscopic constitutive relation through the operator $L^{-1}$. This is consistent with the classical microscopic interpretation of transport coefficients while making the operator pathway explicit \cite{ChapmanCowling,Cercignani}.

The BGK example demonstrates the construction in closed form. Because the inverse collision operator is explicit, the map $\mathcal{K}$ can be evaluated directly and gives the familiar Newtonian stress. The value of the example is therefore not novelty of the BGK coefficient itself, but a transparent realization of the general factorization and its dependence on the collisional relaxation scale.

{The same factorization applies to heat flux by replacing the deviatoric second-moment projection with the heat-flux moment
\begin{multline}
\mathcal{M}_3^{\mathrm{heat}}[g]
:=\int_{\mathbb{R}^d}\frac12|v-u|^2(v-u)g\,dv,
\\\qquad
\mathcal{K}_q:=\mathcal{M}_3^{\mathrm{heat}}L^{-1}(I-\Pi),
\end{multline}
so that $q^{(1)}=\mathcal{K}_q[\mathcal{S}f^{(0)}]$. For the full Boltzmann operator, $L$ is its linearization about the local Maxwellian and $\Pi$ projects onto the collision invariants, as in the present formulation. For ES-BGK, the same structure uses the linearized ES-BGK collision operator instead; its model-specific inverse changes the transport coefficients, including the Prandtl number. In either case, the corresponding solvability and bounded-inverse estimates must be verified for that collision model.}

Several limitations should be emphasized. First, the theorem is conditional on the regularity, solvability, coercivity, and remainder assumptions stated explicitly. Second, the treatment is deterministic and excludes boundary layers, externally forced shear, and stochastic forcing from the main theorem. Third, the bounded kinetic-to-stress map is a first-order constitutive result; it does not by itself establish nonlinear transition, turbulent cascade dynamics, or a direct molecular origin for any particular macroscopic instability. Finally, the microscopic-to-kinetic bridge remains an open modelling and mathematical problem within the broader research program.

{These limitations point to further work.} The operator chain can be generalized to include heat flux, higher-order Burnett corrections, boundary-layer operators, or stochastic kinetic forcing. In a sufficiently controlled setting, the molecular collision statistics developed in the preceding work could then be coupled to a kinetic perturbation model and propagated through $L^{-1}$ and the relevant moment maps. Such a construction would move beyond a qualitative statement about ``molecular seeds'' toward an explicit, testable multiscale transport framework.

\section{Conclusion}

This paper formulates the first-order Chapman--Enskog constitutive step as an explicit kinetic-to-macroscopic transport map. Under stated functional-analytic hypotheses,
\begin{equation}
{\mathcal{K}=\mathcal{M}_2^{\mathrm{dev}}\mathcal{R}_L}
\end{equation}
{and applying this map to $\mathcal{S}f^{(0)}$ gives the bounded composition from the streaming departure of a local Maxwellian, through the collisional kinetic response, to the macroscopic deviatoric stress.} The pressure tensor therefore takes the form
\begin{equation}
P=pI+\varepsilon\mathcal{K}[\mathcal{S}f^{(0)}]+O(\varepsilon^2),
\end{equation}
with the remainder estimate separating the first-order constitutive contribution from higher-order corrections.

The necessity statement follows immediately from this structure: a nonzero first-order deviatoric stress requires a nonzero first-order kinetic correction. More importantly, the operator factorization identifies where the microscopic collision physics enters the kinetic-to-macroscopic passage, without claiming a new viscosity formula or a complete microscopic derivation of hydrodynamics.

Placed alongside the author's preceding collision-statistical work, the present result supplies the kinetic-to-macroscopic half of a broader multiscale program. The preceding studies describe collision-driven finite-scale fluctuations and their stochastic hydrodynamic response; the present analysis specifies how kinetic nonequilibrium is converted into continuum stress. Establishing a rigorous bridge from the molecular fluctuation statistics to a particular kinetic correction remains the next step required for a complete microscopic-to-macroscopic transport theory.

\appendix

\section{Pseudoinverse construction: full statement and proof sketch}\label{app:pseudoinverse}

This appendix provides a detailed statement and proof sketch of the pseudoinverse construction used in the main text.

\begin{proposition}[Pseudoinverse construction — full]\label{prop:appendix-pseudoinverse}

Let $M(x,v)=M[\rho(x),u(x),T(x)]$ with $\rho,u,T\in H^{s}_x$, $s>d/2$, uniformly bounded and sufficiently smooth. Assume $L_x=DQ[M(x,\cdot)]$ satisfies the uniform spectral-gap condition: there exists $\lambda_0>0$ such that for all $x$ and all $g\in N(L_x)^\perp$,

\begin{equation}\label{eq:appendix-coercivity}
{\langle -L_x g,g\rangle_{L^2_v(M(x,\cdot)^{-1})}}\ge\lambda_0\|g\|_{L^2_v(M^{-1})}^2.
\end{equation}

Then for each $x$ the restriction $L_x|_{N(L_x)^\perp}$ is invertible and the pseudoinverse

\begin{equation}\label{eq:Lxinv}
L_x^{-1}:N(L_x)^\perp\to N(L_x)^\perp
\end{equation}

is bounded with

\begin{equation}\label{eq:Lxinv-bound}
\|L_x^{-1}h\|_{L^2_v(M^{-1})}\le C_{\mathrm{inv}}\|h\|_{L^2_v(M^{-1})},
\end{equation}

where $C_{\mathrm{inv}}$ depends only on $\lambda_0$ and uniform bounds on $M$ (and not on $\varepsilon$). Additionally, assume the operator family has the common-domain or resolvent continuity in $x$ needed for $x\mapsto L_x^{-1}$ to be continuous in the chosen operator topology; continuity of $\rho,u,T$ alone does not guarantee operator-norm continuity for every unbounded collision operator.

\end{proposition}

\begin{proof}[Proof sketch]\label{proof:pseudoinverse}

I give a constructive argument for the spectral-gap case and indicate references for the operator results used. When A5 instead uses a stronger-space coercivity estimate, the corresponding inverse bound must be established in that model's stated spaces, as required separately by A6.

Step 1: Frozen-$x$ operator and spectral gap. Fix $x_0$. Consider the frozen operator $L_{x_0}$ acting on $L^2_v(M(x_0,\cdot)^{-1}dv)$. By hypothesis the spectral gap $\lambda_0>0$ on $N(L_{x_0})^\perp$ holds; therefore the restriction $L_{x_0}:N(L_{x_0})^\perp\to N(L_{x_0})^\perp$ is boundedly invertible. This is a standard consequence of coercivity; see Baranger \& Mouhot (theorem 1 of \cite{BarangerMouhot}) or Mouhot \& Strain \cite{MouhotStrain06} for constructive statements in kinetic settings.

Step 2: Uniformity in $x$ via the assumed uniform gap. The Maxwellian depends on the macroscopic fields, and the uniform spectral gap together with uniform bounds on $M$ gives the inverse bound in the stated velocity space. Continuity of the inverse additionally requires the common-domain or resolvent-continuity hypotheses described in A6; continuity of the fields by itself is not asserted to imply operator-norm continuity for an unbounded $L$. Perturbation results for specific Boltzmann operators are discussed by Mouhot \& Strain \cite{MouhotStrain06} and Guo \cite{Guo2002}.

Step 3: Construction of $L_x^{-1}$ on $N(L_x)^\perp$. For each $x$ define the orthogonal projector $\Pi_x$ onto $N(L_x)$ in $L^2_v(M(x,\cdot)^{-1}dv)$ and write $L_x = L_x|_{N(L_x)^\perp}\oplus 0$ on the decomposition. The inverse is given on $N(L_x)^\perp$ by Riesz theory; the coercivity lower bound provides $\|L_x^{-1}\|\le\lambda_0^{-1}$ modulo norm equivalence constants depending on weights (again see \cite{MouhotStrain06}).

Step 4: Continuity in the chosen operator topology. Under the common-domain or resolvent-continuity hypothesis in A6, the resolvent identity gives continuity of $x\mapsto L_x^{-1}$; this is an additional model-dependent condition rather than a consequence of the spectral gap alone.

Remark on hypocoercivity: Hypocoercivity concerns decay of the full evolution with transport and collision; it does not by itself provide a bounded pseudoinverse of the velocity-space operator $L$. Such an inverse must be established separately to satisfy A6. Hypocoercive estimates may instead enter the conditional remainder closure in Proposition~\ref{thm:remainder}.

\end{proof}

\section{Conditional remainder closure}\label{app:conditional-remainder}

For completeness, the remainder mechanism used in Proposition~\ref{thm:remainder} is summarized here. Writing
\[
R_\varepsilon=R_\varepsilon^{\parallel}+R_\varepsilon^{\perp},
\qquad
R_\varepsilon^{\parallel}=\Pi R_\varepsilon,
\quad
R_\varepsilon^{\perp}=(I-\Pi)R_\varepsilon,
\]
the orthogonal component satisfies an energy inequality of the schematic form
\begin{equation}
\frac{1}{2}\frac{d}{dt}\|R_\varepsilon^{\perp}\|^2
+\frac{\lambda_0}{\varepsilon}\|R_\varepsilon^{\perp}\|^2
\le C_1\|R_\varepsilon\|^2+C_2\varepsilon^4,
\end{equation}
provided the commutator and residual estimates required by the chosen kinetic model hold. The parallel component is finite dimensional in velocity and is controlled through the moment equations and the assumed $O(\varepsilon^2)$ initial compatibility. For sufficiently small $\varepsilon$, dissipation absorbs the orthogonal error and a Gr\"onwall estimate gives
\[
\|R_\varepsilon(t)\|_{L^2_xL^2_v(M^{-1})}\le C_T\varepsilon^2.
\]
This is the standard structure of a Chapman--Enskog remainder argument; the present paper uses it only conditionally, as stated in Proposition~\ref{thm:remainder}.

\end{document}